\RequirePackage{fix-cm}
\documentclass[smallextended]{svjour3}       
\smartqed  
\usepackage{graphicx}

\usepackage{amsmath,amsfonts,amssymb,amsthm}
\usepackage{mathtools}
\usepackage[justification=raggedright]{caption}
\usepackage{url}
\usepackage{xcolor}
\theoremstyle{plain}
\theoremstyle{definition}
\theoremstyle{remark}

\begin{document}

\title{Radial expansion preserves hyperbolic convexity and radial contraction preserves spherical convexity. 
}
\subtitle{}


\author{Dhruv Kohli \and
        Jeffrey M. Rabin
}


\institute{Department of Mathematics\\
            University of California, San Diego\\
            La Jolla, CA 92093\\
            \email{\{dhkohli, jrabin\}@ucsd.edu}
            }

\date{February 2019}

\maketitle

\begin{abstract}
On a flat plane, convexity of a set is preserved by both radial expansion and contraction of the set about any point inside it. Using the Poincar\'e disk model of hyperbolic geometry, we prove that radial expansion of a hyperbolic convex set about a point inside it always preserves hyperbolic convexity. Using stereographic projection of a sphere, we prove that radial contraction of a spherical convex set about a point inside it, such that the initial set is contained in the closed hemisphere centred at that point, always preserves spherical convexity.
\keywords{Preserving hyperbolic and spherical convexity, Poincar\'e disk, stereographic projection, dilation, radial expansion and contraction.}
\end{abstract}


\section{Introduction}
Hyperbolic, spherical, and of course Euclidean convexity have been extensively studied as the constant curvature cases of geodesic convexity. Of course convexity is preserved by isometries in all cases, and by dilations as well in Euclidean geometry. In this work we define hyperbolic and spherical analogs of dilations and prove that they preserve convexity under appropriate hypotheses. See \cite{afbeardon1,afbeardon2,maminda} for the basic facts of hyperbolic and spherical geometry used here. Our notation is broadly consistent with these sources.

\medskip

\textbf{Euclidean Plane} The Euclidean plane is $\mathbb{C}$, the complex plane, with Euclidean metric $|dz|$ which is flat; that is, has curvature zero. The translation map
\begin{align}
    \tau_c(z)=z+c \nonumber
\end{align}
is a Euclidean isometry that moves the origin to $c$.

\medskip

\textbf{Euclidean dilations} Euclidean dilation about the origin with dilation factor $k>0$ is given by
\begin{align}
    \delta_{0,k}(z)=kz. \nonumber
\end{align}
Euclidean dilation about a point $c \in \mathbb{C}$ is then given by $\delta_{c,k}=\tau_c \circ \delta_{0,k} \circ \tau_{c}^{-1}$. Obviously, $\delta_{c,k}$ preserves convexity for all $c$ and for all $k>0$. This is what we generalize in hyperbolic and spherical geometry. This overall pattern of translation, dilation about the origin and then translation again will be used in this work to define dilation about an arbitrary point.

\medskip

\textbf{Hyperbolic plane} Let $\mathbb{D} \subseteq \mathbb{C}$ denote the Poincar\'e disk of unit radius with metric $2|dz|/(1-|z|^2)$ which has curvature $-1$. The hyperbolic distance of $z \in \mathbb{D}$ from the origin is given by
\begin{align}
    d(z) = 2\tanh^{-1}\left|z\right|.\nonumber
\end{align}
Given distinct $u,v \in \mathbb{D}$, there is a unique hyperbolic geodesic segment, denoted $[u,v]$, joining these points. Hyperbolic geodesics are arcs of Euclidean circles orthogonal to the unit circle, including Euclidean lines through the origin. A Euclidean circle $|z-a|=r$ is orthogonal to the unit circle $|z|=1$ if and only if $|a|^2=1+r^2$. The map
\begin{align}
    \tau_{c}^{h}(z) = \frac{z+c}{1+\bar{c}z} \nonumber
\end{align}
is the unique hyperbolic isometry that maps the origin to $c$ and has positive derivative at the origin; in fact, $(\tau_c^h)'(0)=1-|c|^2$.
\begin{definition}
    \label{rhconvex}
    A set $C \subseteq \mathbb{D}$ is hyperbolic convex (h-convex) if, for every $u,v \in C$, $[u,v]$ lies in $C$. Obviously, $C$ is h-convex if and only if $\tau_{c}^{h}(C)$ is h-convex for each $c \in \mathbb{D}$.
\end{definition}

\medskip

\textbf{Hyperbolic dilations} Given a direction $e^{i\theta}$, there is a unique hyperbolic geodesic ray $\gamma_0(\theta)$ emanating from the origin with tangent vector $e^{i\theta}$ at the origin. For $z\in\gamma_0(\theta)$, its dilated image is the unique point $\delta^h_{0,k}(z)=z'$ on this ray with $d(z')=kd(z)$. If $z=re^{i\theta}$, then $z'=r'e^{i\theta}$, where $r'=\tanh(k\tanh^{-1}r)$. Thus,
\begin{align}
    \label{rdeltah0}
    \delta^h_{0,k}(re^{i\theta})=\tanh(k\tanh^{-1}r)e^{i\theta}.
\end{align}
We abbreviate $\tau_c^h$ by $\tau_c$ to simplify notation. Then dilation about a point $c \in \mathbb{D}$ is given by $\delta^h_{c,k} = \tau_c \circ \delta^h_{0,k} \circ \tau^{-1}_c$.

\medskip

\textbf{Spherical plane} The spherical plane is the one-point compactification $\hat{\mathbb{C}} = \mathbb{C} \cup \{\infty\}$ of the complex plane with the metric $2|dz|/(1+|z|^2)$ which has curvature $+1$. The extended plane $\hat{\mathbb{C}}$ is isometric under stereographic projection to the unit sphere $\mathbb{S}$ in $\mathbb{R}^3$ with its metric as a subset of $\mathbb{R}^3$. A Euclidean disk $D$ or Euclidean half-plane $H$ in $\mathbb{C}$ is called a \textit{hemisphere} if its stereographic projection onto $\mathbb{S}$ is a hemisphere. A disk $D_s(c,r)$ is a hemisphere if and only if $1+|c|^2=r^2$. The spherical distance is
\begin{align}
    \hat{d}(z,w)=2\tan^{-1}\left|\frac{z-w}{1+\bar{w}z}\right|,\nonumber
\end{align}
with obvious changes if one of the points is $\infty$. Points $u,v \in \hat{C}$ are \textit{antipodal} when $v = -1/\bar{u}$; this is equivalent to $\hat{d}(u,v)=\pi$. If $u,v$ are not antipodal, then there is a unique spherical geodesic segment $[u,v]$ joining the points; it is the shorter arc of the unique great circle through the points. The map
\begin{align}
    \tau^s_c(z) = \frac{z+c}{1-\bar{c}z} \nonumber
\end{align}
is the unique spherical isometry that takes the origin to $c$ with positive derivative at the origin: $(\tau^{s}_c)'(0)=1+|c|^2>0$. It is a rotation when viewed as acting on the unit sphere.
\begin{definition}
    \label{rsconvex}
    A set $C \subseteq \hat{\mathbb{C}}$ is spherical convex (s-convex) if, for every $u, v \in C$, all spherical geodesic segments joining them lie in C. Naturally, $C$ is s-convex if and only if $\tau_{c}^{s}(C)$ is s-convex for each $c \in \mathbb{D}$.
\end{definition}
With this definition, $\hat{\mathbb{C}}$ is s-convex. If a s-convex set $C$ contains a pair of antipodal points, then the set must be $\hat{\mathbb{C}}$. If a s-convex set does not contain antipodal points then the set is contained in a hemisphere centred at $a$, $\overline{D}_s(a,\pi/2)=\{z \in \hat{\mathbb{C}}: \hat{d}(a,z)\leqslant \pi/2\}$ for some $a \in C$. We will always assume this is the case.

\medskip

\textbf{Spherical dilations} Given $e^{i\theta}$, there is a unique spherical geodesic ray $\gamma_0(\theta)$ emanating from the origin with the tangent vector $e^{i\theta}$ at the origin. For $z \in \gamma_0(\theta)$, its dilated image is the unique point $\delta^s_{0,k}(z)=z'$ on this ray with $\hat{d}(0,z')=k\hat{d}(0,z)$ provided $k\hat{d}(0,z) < \pi$. If $z=re^{i\theta}$, then $z'=r'e^{i\theta}$, where $r'=\tan(k\tan^{-1}r)$. Thus,
\begin{align}
    \label{rdeltas0}
    \delta^s_{0,k}(re^{i\theta})=\tan(k\tan^{-1}r)e^{i\theta}.
\end{align}

Again abbreviating $\tau^s_c$ by $\tau_c$, the spherical dilation about a point $c \in \hat{\mathbb{C}}$ is given by $\delta^s_{c,k} = \tau_c \circ \delta^s_{0,k} \circ \tau_c^{-1}$.

\medskip

For all three geometries, we refer to the dilation of a set as \textit{radial expansion} when $k\geqslant 1$ and as \textit{radial contraction} when $k\leqslant 1$. Henceforth, expansion always means radial expansion and contraction always means radial contraction, hyperbolic or spherical according to context.

\section{Results}
\label{sec:result}
It is clear that in Euclidean geometry, if a set $C$ is convex then irrespective of whether dilation is an expansion or a contraction, the dilated set $C'$ is also convex. In this work, we prove that if $C$ is h-convex then the set obtained by expansion of $C$ about a point in it is still h-convex. We also prove that if $C$ is s-convex then the set obtained by  contraction of $C$ about a point $c$ in $C$ is still s-convex, provided that the initial set $C$ is contained in $\overline{D}_s(c,\pi/2)$, the closed spherical half-plane centred at $c$. We prove the hyperbolic case first.

\begin{theorem}
    \label{thm1}
    Consider a h-convex set $C \subseteq \mathbb{D}$ and a point $c \in C$. Then for any $k\geqslant 1$, $\delta^h_{c,k}(C)$ is h-convex.
\end{theorem}
\begin{proof}
    Because of the invariance of h-convexity and dilations under isometries of the hyperbolic metric, there is no loss of generality in assuming that $c = 0$. So, we prove that for any $k \geqslant 1$, $\delta^h_{0,k}(C)$ is h-convex when $C$ is h-convex and $0\in C$. Denote $\delta^h_{0,k}(C)$ by $C'$.

    \medskip
    
    The following is an outline of the proof. We take two arbitrary points $x_1', x_2' \in C'$. We then take an arbitrary point $x'$ on $[x_1',x_2']$. To show that $C'$ is h-convex, we must show that $x' \in C'$ (Definition (\ref{rhconvex})). The preimages of $x_1', x_2'$ and $x'$ under the map $\delta^h_{0,k}$ are then computed. We call these preimages $x_1, x_2$ and $x$. Since $x_1',x_2' \in C'$, therefore $x_1, x_2 \in C$. We must show that $x \in C$ which will then prove that $x' = \delta^h_{0,k}(x) \in C'$. To show that $x \in C$, we first find a point $\xi$ on $[x_1,x_2]$ which is on the same radial geodesic ray as $x$. Note that $\xi \in C$ because $C$ is h-convex. Then we show that $|\xi| \geqslant |x|$. Using this and the facts that $0, \xi \in C$ and $C$ is h-convex, we conclude that $x \in C$ and therefore $C'$ is h-convex.

    \medskip

    For convenience, denote $s = 1/k$, so $s \in (0,1]$. Consider two points $x_1', x_2' \in C'$ as follows,
    \begin{align}
    x_1' &= r_1'e^{i\theta_1},\nonumber\\
    x_2' &= r_2'e^{i\theta_2},\nonumber
    \end{align}
    where $r_1', r_2' \in (0,1)$. Without loss of generality, assume that $0 \leqslant  \theta_1 < \theta_2 < \pi$. For convenience, denote $\gamma_1 = \tanh^{-1}r_1'$ and $\gamma_2 = \tanh^{-1}r_2'$. Rewriting $x_1'$ and $x_2'$,
    \begin{align}
    x_1' &= \tanh(\gamma_1)e^{i\theta_1},\nonumber\\
    x_2' &= \tanh(\gamma_2)e^{i\theta_2}.\nonumber
    \end{align}
    Let $x'$ be a point on $[x_1',x_2']$. We can represent $x'$ as
    \begin{align}
    x' &= r'e^{i\lambda}, \nonumber
    \end{align}
    where $\lambda = \theta_1 + t(\theta_2 - \theta_1)$ for some $t \in (0,1)$. Later we will obtain the value of $r'$ in terms of $\gamma_1,\gamma_2,\theta_1,\theta_2$ and $\lambda$ (equation (\ref{atanhrp})). Using the inverse of $\delta^h_{0,k}$, (equation (\ref{rdeltah0})), we obtain $x_1, x_2, x$ from $x_1', x_2', x'$ as
    \begin{align}
    x_1 &= r_1e^{i\theta_1},\nonumber\\
    x_2 &= r_2e^{i\theta_2},\nonumber\\
    x &= re^{i\lambda},\nonumber
    \end{align}
    where
    \begin{align}
    \label{r_1}
    r_1 &= \tanh(s\tanh^{-1}r_1') = \tanh(\gamma_1s),\\
    \label{r_2}
    r_2 &= \tanh(s\tanh^{-1}r_2') = \tanh(\gamma_2s),\\
    \label{r}
    r &= \tanh(s\tanh^{-1}r').
    \end{align}
    Then $[x_1,x_2]$ will be an arc of a circle $K$ orthogonal to $\partial \mathbb{D}$ centred at $a = a_1 + ia_2$ with radius $R$. Since $x_1$ and $x_2$ lie on $K$ and $|a|^2=1+R^2$, we obtain,
    \begin{align}
    \label{a_1}
    a_1 &= \frac{(r_1^{-1}+r_1)\sin\theta_2 -(r_2^{-1}+r_2)\sin\theta_1}{2\sin(\theta_2-\theta_1)}\nonumber\\
    &= \frac{\coth(2\gamma_1s)\sin\theta_2 - \coth(2\gamma_2s)\sin\theta_1}{\sin(\theta_2-\theta_1)},\\
    \label{a_2}
    a_2 &= \frac{(r_2^{-1}+r_2)\cos\theta_1 -(r_1^{-1}+r_1)\cos\theta_2}{2\sin(\theta_2-\theta_1)}\nonumber\\
    &= \frac{\coth(2\gamma_2s)\cos\theta_1 - \coth(2\gamma_1s)\cos\theta_2}{\sin(\theta_2-\theta_1)}.
    \end{align}
    We used equations (\ref{r_1}, \ref{r_2}) and the identity $\tanh(\alpha) + \coth(\alpha) = 2\coth(2\alpha)$ to simplify the above expressions. 

    \medskip
    
    Since $x_1', x_2' \in C'$, we have $x_1,x_2 \in C$. Since $C$ is h-convex, every point on $[x_1,x_2]$ lies in $C$. There exists a point on $[x_1,x_2]$ in the direction $e^{i\lambda}$ of $x$. Let that point be $\xi$ and denote $\rho = |\xi|$. Then,
    \begin{align}
    \rho &= a_1\cos\lambda + a_2\sin\lambda -\sqrt{(a_1\cos\lambda  + a_2\sin\lambda)^2 - 1}.\nonumber
    \end{align}
    Note that the circle $K$ intersects the line passing through $0$ and $x$ at two points, one of which lies inside $\mathbb{D}$, the other outside. The above equation ensures that $\xi$ is the point of intersection which lies inside $\mathbb{D}$. Also, note that
    {
    \begin{align}
        a_1\cos\lambda + a_2\sin\lambda &= \frac{\coth(2\gamma_1s)\sin(\theta_2-\lambda) +\coth(2\gamma_2s)\sin(\lambda-\theta_1)}{\sin(\theta_2-\theta_1)} \nonumber\\
        &\geqslant \frac{\sin(\theta_2-\lambda) +\sin(\lambda-\theta_1)}{\sin(\theta_2-\theta_1)}\nonumber\\
        &\geqslant 1.\nonumber
    \end{align}
    }%
    Using $\tanh^{-1}(\alpha - \sqrt{\alpha^2-1}) = \frac{1}{2}\tanh^{-1}\left(\frac{1}{\alpha}\right)$, we get
    {
    \begin{align}
        \tanh^{-1}\rho &= \frac{1}{2}\tanh^{-1}\left(\frac{1}{a_1\cos\lambda  + a_2\sin\lambda }\right) .\nonumber
    \end{align}
    }%
    Substituting the values of $a_1$ and $a_2$ using equations (\ref{a_1}, \ref{a_2}), we get
    {\small
    \begin{align}
        \label{atanhrho}
        \tanh^{-1}\rho= \frac{1}{2}\tanh^{-1}\left(\frac{\sin(\theta_2-\theta_1)}{\coth(2\gamma_1s)\sin(\theta_2-\lambda) +\coth(2\gamma_2s)\sin(\lambda-\theta_1)}\right).
    \end{align}
    }%
    Also, note that $r = |x|$ and from equation (\ref{r}) we have $\tanh^{-1}r = s\tanh^{-1}r'$. Note that when $k = 1$ (equivalently $s=1$) we have $x_1' = x_1, x_2' = x_2$ and $x' = x = \xi$. So, by equating $\tanh^{-1}r$ and $\tanh^{-1}\rho$ at $s = 1$, we obtain
    {
    \begin{align}
        \label{atanhrp}
        \tanh^{-1}r' &= \frac{1}{2}\tanh^{-1}\left(\frac{\sin(\theta_2-\theta_1)}{\coth(2\gamma_1)\sin(\theta_2-\lambda) + \coth(2\gamma_2)\sin(\lambda-\theta_1)}\right).
    \end{align}
    }%
    Therefore,
    {
    \begin{align}
    \label{atanhr}
    \tanh^{-1}r &= \frac{s}{2}\tanh^{-1}\left(\frac{\sin(\theta_2-\theta_1)}{\coth(2\gamma_1)\sin(\theta_2-\lambda) + \coth(2\gamma_2)\sin(\lambda-\theta_1)}\right).
    \end{align}
    }%
    Using equations (\ref{atanhrho}, \ref{atanhr}), we show that $\rho \geqslant r$ by showing that $\tanh^{-1}\rho \geqslant \tanh^{-1}r$ (because $\rho, r \in (0,1)$). Note that $\tanh^{-1}r$ is linear in $s$ and $\tanh^{-1}\rho$ is concave in $s$ for all $s > 0$ (using Lemma (\ref{lem3}) in the Appendix). Also, note that in the limit $s \rightarrow 0$,  $\tanh^{-1}r = \tanh^{-1}\rho = 0$, and at $s = 1$, $\tanh^{-1}r = \tanh^{-1}\rho = \tanh^{-1}r'$. These constraints on $\tanh^{-1}r$ and $\tanh^{-1}\rho$ imply that $\tanh^{-1}\rho \geqslant \tanh^{-1}r$ for all $s \in (0,1]$.

    \medskip

    Therefore, $\rho \geqslant r$ and so $|\xi| \geqslant |x|$. Using this and the facts that $0, \xi \in C$, $x$ is on the same radial geodesic ray $[0,\xi]$, and $C$ is h-convex, we conclude that $x \in C$. So, $x' \in C'$ and $C'$ is h-convex.
\end{proof}

\medskip

The spherical case is similar except that the initial s-convex set is restricted to be contained in the spherical half-plane centred at the point about which the set is dilated.

\begin{theorem}
    \label{thm2}
    Consider a s-convex set $C \subseteq \hat{\mathbb{C}}$ and a point $c \in C$ such that $C \subseteq \overline{D}_s(c,\pi/2)$. Then for any $0 < k \leqslant 1$, $\delta^s_{c,k}(C)$ is s-convex.
\end{theorem}
\begin{proof}
    Because of the invariance of s-convexity and s-dilations under spherical isometries, there is no loss of generality in assuming that $c=0$. So, we prove that for any $k \in (0,1]$, $\delta^s_{0,k}(C)$ is s-convex when $C \subseteq \overline{\mathbb{D}} = \overline{D}_s(0,\pi/2)$ is s-convex and $0\in C$. Denote $\delta^s_{0,k}(C)$ by $C'$. Note that $C' \subseteq C \subseteq \overline{\mathbb{D}}$. The strategy of the proof is same as that of Theorem \ref{thm1}.

    \medskip

    For convenience, denote $s = 1/k$, so that $s \geqslant 1$. Later we will obtain an upper bound on $s$ based on the constraint $C \subseteq \overline{\mathbb{D}}$. Consider two points $x_1', x_2' \in C'$,
    \begin{align}
    x_1' &= r_1'e^{i\theta_1},\nonumber\\
    x_2' &= r_2'e^{i\theta_2},\nonumber
    \end{align}
    where $r_1', r_2' \in (0,1)$. Without loss of generality, assume that $0 \leqslant  \theta_1 < \theta_2 < \pi$. For convenience, denote $\gamma_1 = \tan^{-1}r_1'$ and $\gamma_2 = \tan^{-1}r_2'$. Rewriting $x_1'$ and $x_2'$,
    \begin{align}
    x_1' &= \tan(\gamma_1)e^{i\theta_1},\nonumber\\
    x_2' &= \tan(\gamma_2)e^{i\theta_2}.\nonumber
    \end{align}
    Let $x'$ be a point on $[x_1',x_2']$. We can represent $x'$ as
    \begin{align}
    x' &= r'e^{i\lambda}, \nonumber
    \end{align}
    where $\lambda = \theta_1 + t(\theta_2 - \theta_1)$ for some $t \in (0,1)$. Using the inverse of $\delta^s_{0,k}$ (equation (\ref{rdeltas0})), we obtain $x_1, x_2, x$ from $x_1', x_2', x'$:
    \begin{align}
    x_1 &= r_1e^{i\theta_1},\nonumber\\
    x_2 &= r_2e^{i\theta_2},\nonumber\\
    x &= re^{i\lambda},\nonumber
    \end{align}
    where
    \begin{align}
    \label{sr_1}
    r_1 &= \tan(s\tan^{-1}r_1') = \tan(\gamma_1s),\\
    \label{sr_2}
    r_2 &= \tan(s\tan^{-1}r_2') = \tan(\gamma_2s),\\
    \label{sr}
    r &= \tan(s\tan^{-1}r').
    \end{align}
    Since $x_1, x_2, x \in \overline{\mathbb{D}}$, therefore $r_1,r_2,r \in (0,1]$, or equivalently, their spherical distance from $0$ is less than or equal to $\pi/2$. Since $x$ lies on $[x_1,x_2]$, therefore the spherical distance of $x$ from $0$ is less than the spherical distance of either $x_1$ or $x_2$ from $0$. So we obtain the following constraint on $s$,
    \begin{align}
        2s\ \max(\gamma_1, \gamma_2)  &\leqslant \frac{\pi}{2}.\nonumber
    \end{align}
    Define $s^*$ as,
    \begin{align}
        s^* = \frac{\pi}{4}\min(\gamma_1^{-1}, \gamma_2^{-1}).\nonumber
    \end{align}
    Note that $s^* \geqslant 1$ and we have $s \in [1,s^*]$.

    \medskip

    Then $[x_1,x_2]$ will be an arc of a circle $K$ which intersects the unit circle at diametrically opposite points, and is centred at $a = (a_1,a_2)$ with radius $R$. Since $x_1$ and $x_2$ lie on $K$ and $1+|a|^2=R^2$, we obtain
    \begin{align}
    \label{sa_1}
    a_1 &= \frac{(r_1-r_1^{-1})\sin\theta_2 -(r_2-r_2^{-1})\sin\theta_1}{2\sin(\theta_2-\theta_1)}\nonumber\\
    &= -\frac{\cot(2\gamma_1s)\sin\theta_2 -\cot(2\gamma_2s)\sin\theta_1}{\sin(\theta_2-\theta_1)},\\
    \label{sa_2}
    a_2 &= \frac{(r_2-r_2^{-1})\cos\theta_1 -(r_1-r_1^{-1})\cos\theta_2}{2\sin(\theta_2-\theta_1)}\nonumber\\
    &= -\frac{\cot(2\gamma_2s)\cos\theta_1 -\cot(2\gamma_1s)\cos\theta_2}{\sin(\theta_2-\theta_1)}.
    \end{align}
    We used equations (\ref{sr_1}, \ref{sr_2}) and the identity $\cot(\alpha) - \tan(\alpha) = 2\cot(2\alpha)$ to simplify these expressions.

    \medskip
    
    Since $x_1', x_2' \in C'$, we have $x_1,x_2 \in C$. Since $C$ is s-convex, every point on $[x_1,x_2]$ lies in $C$. There exists a point on $[x_1,x_2]$ which has the same direction $e^{i \lambda}$ as $x$ because $\lambda \in (\theta_1, \theta_2)$. Let that point be $\xi$ and denote $\rho = |\xi|$. Then,
    \begin{align}
    \rho &= \sqrt{(a_1\cos\lambda  + a_2\sin\lambda)^2 + 1} + a_1\cos\lambda + a_2\sin\lambda.\nonumber
    \end{align}
    The above equation ensures that $\xi$ is the point of intersection which lies inside the unit circle. Also, note that
    {
    \begin{align}
        a_1\cos\lambda  + a_2\sin\lambda &= -\frac{\cot(2\gamma_1s)\sin(\theta_2-\lambda) +\cot(2\gamma_2s)\sin(\lambda-\theta_1)}{\sin(\theta_2-\theta_1)} \nonumber\\
        &\leqslant 0. \nonumber
    \end{align}
    }%
    
    Using $\tan^{-1}(\alpha + \sqrt{\alpha^2+1}) = \frac{1}{2}\tan^{-1}\left(\frac{-1}{\alpha}\right)$ when $\alpha \leqslant 0$, we get
    {
    \begin{align}
        \tan^{-1}\rho &= \frac{1}{2}\tan^{-1}\left(\frac{-1}{a_1\cos\lambda  + a_2\sin\lambda }\right). \nonumber
    \end{align}
    }%
    Substituting the values of $a_1$ and $a_2$ using equations (\ref{sa_1}, \ref{sa_2}) we get
    {
    \begin{align}
        \label{atanrho}
        \tan^{-1}\rho= \frac{1}{2}\tan^{-1}\left(\frac{\sin(\theta_2-\theta_1)}{\cot(2\gamma_1s)\sin(\theta_2-\lambda) +\cot(2\gamma_2s)\sin(\lambda-\theta_1)}\right).
    \end{align}
    }%
    Setting $s=1$ we obtain
    {
    \begin{align}
        \label{atanrp}
        \tan^{-1}r' &= \frac{1}{2}\tan^{-1}\left(\frac{\sin(\theta_2-\theta_1)}{\cot(2\gamma_1)\sin(\theta_2-\lambda) + \cot(2\gamma_2)\sin(\lambda-\theta_1)}\right).
    \end{align}
    }%
    Therefore,
    {
    \begin{align}
    \label{atanr}
    \tan^{-1}r &= \frac{s}{2}\tan^{-1}\left(\frac{\sin(\theta_2-\theta_1)}{\cot(2\gamma_1)\sin(\theta_2-\lambda) + \cot(2\gamma_2)\sin(\lambda-\theta_1)}\right).
    \end{align}
    }%
    Using equations (\ref{atanrho}, \ref{atanr}), we show that $\rho \geqslant r$ by showing that $\tan^{-1}\rho \geqslant \tan^{-1}r$ (because $\rho, r \in (0,1)$). Note that $\tan^{-1}r$ is linear in $s$ and $\tan^{-1}\rho$ is convex in $s$ for all $s \in (0,s^*]$ (using Lemma (\ref{lem4}) in the Appendix). Also, note that in the limit $s \rightarrow 0$,  $\tan^{-1}r = \tan^{-1}\rho = 0$, and at $s = 1$, $\tan^{-1}r = \tan^{-1}\rho = \tan^{-1}r'$. These constraints on $\tan^{-1}r$ and $\tan^{-1}\rho$ imply that $\tan^{-1}\rho \geqslant \tan^{-1}r$ for all $s \in [1,s^*]$.

    \medskip

    Therefore, $\rho \geqslant r$ and so $|\xi| \geqslant |x|$. As before this shows that $x' \in C'$ and $C'$ is s-convex.
\end{proof}

\medskip

We now provide examples which show that the hypotheses of our theorems are necessary.
\begin{itemize}
    \item  Contraction of a h-convex set about a point in it may not preserve h-convexity. Consider a hyperbolic geodesic $\gamma$ that does not contain the origin. Let $H$ be the closed hyperbolic half-plane determined by $\gamma$ that contains the origin. Consider $\delta^h_{0,k}(H)$, where $k \in (0,1)$. Note that $\delta^h_{0,k}(\gamma)$ is a curve in $H$ that has the same endpoints, say $a$ and $b$, on the circle as $\gamma$. If one selects two points on $\delta^h_{0,k}(\gamma)$ that are very near $a$ and $b$, respectively, then the hyperbolic geodesic through these points is very close to $\gamma$. Because $\delta^h_{0,k}(\gamma)$ lies in the interior of $H$, this hyperbolic geodesic must contain points outside $\delta^h_{0,k}(H)$, so $\delta^h_{0,k}(H)$ is not h-convex.
    \item Expansion or contraction of a h-convex or a s-convex set $C$ about a point outside it may not preserve h-convexity or s-convexity. This follows directly from the fact that in both hyperbolic and spherical geometries, dilation of a geodesic segment about a point outside it results in a segment which is not a geodesic, so simply take $C$ to be such a geodesic segment.
    \item Expansion of a s-convex set $C$ about $0$, where $0 \in C$, may not preserve s-convexity. Consider a geodesic segment $C=\gamma$ passing through $0$ and having length slightly less than $\pi$ with length approximately $\pi/2$ on either side of $0$. Clearly, $\gamma$ is s-convex. With a sufficiently large dilation factor $k \gg 1$, $\delta^h_{0,k}(\gamma)$ will be a geodesic containing at least two antipodal points. Since $\hat{\mathbb{C}}$ is the only s-convex set containing antipodal points, $\delta^s_{0,k}(\gamma)$ is not s-convex.
    \item Contraction of a s-convex set $C$ about $0$ when $C \not\subseteq \overline{\mathbb{D}}$ and $0 \in C$, may not preserve s-convexity. Consider the s-convex hull $C$ of the points $0$, $\tan(\frac{0.9\pi}{2}) e^{i\pi/6}$ and $\tan(\frac{0.9\pi}{2}) e^{i\pi/3}$. Clearly, $C$ is s-convex and $C \not\subseteq \overline{\mathbb{D}}$. We then take a dilation factor of $0.9$ and plot $C' = \delta^{s}_{0,0.9}(C)$ as well as the s-convex hull of $C'$ (Figure (\ref{fig1})). Clearly, the s-convex hull of $C'$ is not contained in $C'$. So, $C'$ is not s-convex.
\end{itemize}
\begin{figure*}[h]
\centering
\includegraphics[width=0.8\textwidth]{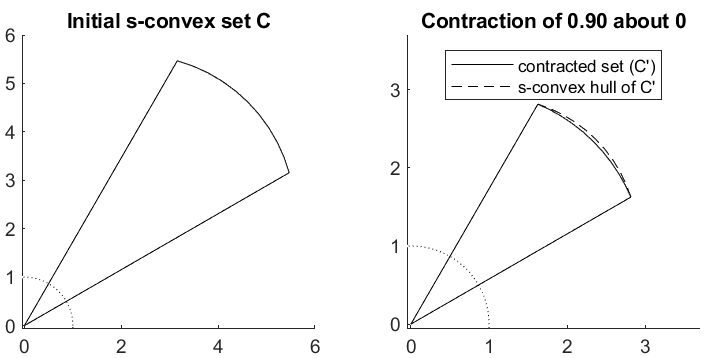}
\caption{Contraction of a s-convex set $C$ about $0$ when $C \not\subseteq \overline{\mathbb{D}}$ and $0 \in C$ resulting in a set which is not s-convex. The dotted quarter circle is the stereographic image of a hemisphere centred at the origin.}
\label{fig1}
\end{figure*}

\section{Conclusion and future work}
In this work, we showed that expansion of a hyperbolic convex set in the Poincar\'e disk about a point inside it results in a hyperbolic convex set while contraction may not. We also showed that contraction of a spherical convex set about a point inside it, such that the set is contained in the closed spherical half-plane centred at that point, results in a spherical convex set while expansion may not. This is in contrast to the case on a flat plane, where both contraction and expansion preserve convexity. Although not proved in this work, we conjecture that our results still hold for asymmetric dilation as well as in higher dimensions. For example, in the planar Euclidean case, asymmetric dilation means scaling by a diagonal matrix having unequal entries, say $k_1$ and $k_2$, so that a point 
$r e^{i \theta}$ maps to $r' e^{i \theta'}$ where 
$r' = r\sqrt{k_1^2\cos^2\theta +k_2^2\sin^2\theta}$ and  
$\tan \theta' = (k_2/k_1)\tan\theta$. Data from computer experiments supports these conjectures, which we hope to prove in our future work. 

\section{Conflict of interest statement}
On behalf of all authors, the corresponding author states that there is no conflict of interest. 

\begin{acknowledgements}
We would like to thank M. Xiao for several useful discussions that helped in defining scaling transformations in the Poincar\'e disk. We also thank the anonymous referee for generous comments that greatly improved this paper.\\

This is a post-peer-review, pre-copyedit version of an article published in Journal of Geometry. The final authenticated version is available online at: \href{https://doi.org/10.1007/s00022-019-0497-8}.
\end{acknowledgements}


\begin{thebibliography}{1}
\providecommand{\url}[1]{{#1}}
\providecommand{\urlprefix}{URL }
\expandafter\ifx\csname urlstyle\endcsname\relax
    \providecommand{\doi}[1]{DOI~\discretionary{}{}{}#1}\else
    \providecommand{\doi}{DOI~\discretionary{}{}{}\begingroup
    \urlstyle{rm}\Url}\fi

\bibitem{afbeardon1}
Beardon, A. F.: The Geometry of Discrete Groups.
\newblock Springer Verlag, New York, 1982.

\bibitem{afbeardon2}
Beardon, A. F., Minda, D.: The hyperbolic metric and geometric function theory, in Quasiconformal Mappings and their Applications, ed. by S. Ponnusamy, T. Sugawa and M. Vuorinen.
\newblock Narosa Publishing House, New Delhi, 2007, 9–56.

\bibitem{maminda} 
Ma, W., Minda, D.: Geometric properties of hyperbolic geodesics. In : Proceedings of the International Workshop on Quasiconformal Mappings and their Applications (2007).

\end{thebibliography}


\appendix\normalsize
\section*{Appendix}
\begin{lemma}
    \label{lem3}
    Let $f(x) = \tanh^{-1}\left(\frac{1}{k_1\coth(u_1x) + k_2\coth(u_2x)}\right)$. If $k_1,k_2 > 0,k_1 + k_2 \geqslant 1$ and $u_1,u_2 > 0$, then $f(x)$ is concave, that is, $f''(x) \leqslant 0$ for all $x > 0$.
\end{lemma}
\begin{proof}
    For convenience, denote $p_i \equiv \coth(u_ix)$. Note that $p_i \geqslant 1$ for all $x \geqslant 0$ and $p_i' \equiv u_i(1-p_i^2)$. Then
    {
    \begin{align}
    f'(x) &= \frac{f_1(x)}{f_2(x)},\nonumber
    \end{align}
    }%
    where
    {
    \begin{align}
        f_1(x) &= k_1u_1(p_1^2-1) + k_2u_2(p_2^2-1),\nonumber\\
        f_2(x) &= (k_1p_1 + k_2p_2)^2-1.\nonumber
    \end{align}
    }%
    Note that $f_1(x), f_2(x) > 0$ for all $x > 0$. Next,
    {
    \begin{align}
        f''(x) &= \frac{f_1'(x)f_2(x)-f_1(x)f_2'(x)}{f_2(x)^2},\nonumber
    \end{align}
    }%
    where
    {
        \begin{align}
            f_1'(x) &= -2(k_1u_1^2p_1(p_1^2-1) + k_2u_2^2p_2(p_2^2-1)),\nonumber\\
            f_2'(x) &= -2(k_1p_1+k_2p_2)f_1(x).\nonumber
        \end{align}
    }%
    Note that $f_1'(x), f_2'(x) \leqslant 0$ for all $x > 0$. The denominator $f_2(x)^2$ is non-negative. We show that the numerator $f_1'(x)f_2(x)-f_1(x)f_2'(x)$ is non-positive. Using the Cauchy-Schwarz inequality, we have
    {\small
    \begin{align}
        f_1(x)^2 &= \left(u_1\sqrt{k_1p_1(p_1^2-1)} \cdot \sqrt{k_1(p_1 - p_1^{-1})} + u_2\sqrt{k_2p_2(p_2^2-1)}\cdot \sqrt{k_2(p_2-p_2^{-1})}\right)^2\nonumber\\
        &\leqslant (k_1u_1^2p_1(p_1^2-1) + k_2u_2^2p_2(p_2^2-1))(k_1(p_1-p_1^{-1})+k_2(p_2-p_2^{-1}))\nonumber\\
        &= -\frac{f_1'(x)}{2}(k_1(p_1-p_1^{-1})+k_2(p_2-p_2^{-1})).\nonumber
    \end{align}
    }%
    Substituting the above inequality in $f_1'(x)f_2(x)-f_1(x)f_2'(x)$, we get
    {
    \begin{align}
        &f_1'(x)f_2(x)-f_1(x)f_2'(x)\nonumber\\
        &\ \ = f_1'(x)f_2(x)+2(k_1p_1+k_2p_2)f_1(x)^2 \nonumber\\
        &\ \ \leqslant f_1'(x)(f_2(x)-(k_1p_1+k_2p_2)(k_1(p_1-p_1^{-1})+k_2(p_2-p_2^{-1}))) \nonumber\\
        &\ \ = f_1'(x)((k_1p_1+k_2p_2)^2 - 1 - (k_1p_1+k_2p_2)(k_1(p_1-p_1^{-1})+k_2(p_2-p_2^{-1}))) \nonumber\\
        &\ \ = f_1'(x)(k_1^2+k_2^2+k_1k_2(p_1p_2^{-1}+p_2p_1^{-1})- 1) \nonumber\\
        \label{exp1}
        &\ \ \leqslant  f_1'(x)((k_1+k_2)^2- 1)\\
        &\ \ \leqslant 0.  \nonumber
    \end{align}
    }%
    Inequality (\ref{exp1}) follows from the facts that $f_1'(x) \leqslant 0$ and $p_1p_2^{-1}+p_2p_1^{-1} \geqslant 2$ for all $x > 0$. The last inequality follows from $f_1'(x) \leqslant 0$ and $(k_1+k_2)^2 \geqslant 1$.
\end{proof}

\begin{lemma}
    \label{lem4}
    Let $f(x) = \tan^{-1}\left(\frac{1}{k_1\cot(u_1x) + k_2\cot(u_2x)}\right)$. If $k_1,k_2 > 0, k_1+k_2\geqslant 1$, and $u_1,u_2 > 0$, then $f(x)$ is convex, that is, $f''(x) \geqslant 0$ for all $x \in (0,x^*]$ where $x^* = \frac{\pi}{2}\min(u_1^{-1},u_2^{-1})$.
\end{lemma}
\begin{proof}
    For convenience, denote $p_i \equiv \cot(u_ix)$. Note that $p_i \geqslant 0$ for all $x \in (0,x^*]$ and $p_i' \equiv -u_i(1+p_i^2)$. Then
    {
    \begin{align}
    f'(x) &= \frac{f_1(x)}{f_2(x)}, \nonumber
    \end{align}
    }%
    where,
    {
    \begin{align}
    f_1(x) &= k_1u_1(1+p_1^2) + k_2u_2(1+p_2^2), \nonumber\\
    f_2(x) &= (k_1p_1 + k_2p_2)^2+1. \nonumber
    \end{align}
    }%
    Note that $f_1(x), f_2(x) > 0$ for all $x \in (0,x^*]$. Next,
    {
    \begin{align}
    f''(x) &= \frac{f_1'(x)f_2(x) - f_1(x)f_2'(x)}{f_2(x)^2},\nonumber
    \end{align}
    }%
    where
    {
    \begin{align}
        \label{f2}
        f_1'(x) &= -2\left(k_1u_1^2p_1(1+p_1^2) + k_2u_2^2p_2(1+p_2^2)\right), \nonumber\\
        f_2'(x) &= -2(k_1p_1+k_2p_2)f_1(x).
    \end{align}
    }%
    Note that $f_1'(x), f_2'(x) \leqslant 0$ for all $x \in (0,x^*]$. The denominator $f_2(x)^2$ is positive, so we will show that the numerator $f_1'(x)f_2(x) - f_1(x)f_2'(x)$ is also positive for all $x \in (0,x^*]$. First we show that
    {
    \begin{align}
        \label{scs2}
        f_1'(x)\left(k_1\sqrt{p_1^2+1} + k_2\sqrt{p_2^2+1}\right)^2 - f_1(x)f_2'(x) \geqslant 0.
    \end{align}
    }%
    Using the Cauchy-Schwarz inequality, we have
    {
    \begin{align}
        &\left(k_1\sqrt{p_1^2+1} + k_2\sqrt{p_2^2+1}\right)^2 \nonumber\\
        &\ \ = \left(\sqrt{k_1u_1(p_1^2+1)}\cdot \sqrt{k_1u_1^{-1}} + \sqrt{k_2u_2(p_2^2+1)}\cdot \sqrt{k_2u_2^{-1}}\right)^2 \nonumber\\
        &\ \ \leqslant  (k_1u_1(p_1^2+1) + k_2u_2(p_2^2+1))(k_1u_1^{-1}+k_2u_2^{-1}) \nonumber\\
        &\ \ = f_1(x)(k_1u_1^{-1}+k_2u_2^{-1}).\nonumber
    \end{align}
    }%
    Using the above inequality and equation (\ref{f2}) in equation (\ref{scs2}), we get
    {
    \begin{align}
        &f_1'(x)\left(k_1\sqrt{p_1^2+1} + k_2\sqrt{p_2^2+1}\right)^2 - f_1(x)f_2'(x)\nonumber\\
        \label{scs31}
        &\ \ \geqslant f_1'(x)f_1(x)(k_1u_1^{-1}+k_2u_2^{-1}) - f_1(x)f_2'(x)\\
        &\ \ = f_1'(x)f_1(x)(k_1u_1^{-1}+k_2u_2^{-1}) + 2(k_1p_1+k_2p_2)f_1(x)^2\nonumber\\
        &\ \ = f_1(x)(f_1'(x)(k_1u_1^{-1}+k_2u_2^{-1}) + 2(k_1p_1+k_2p_2)f_1(x))\nonumber\\
        \label{scs3}
        &\ \ = 2f_1(x)\left[\frac{k_1k_2}{u_1u_2}\left(u_1p_1-u_2p_2\right)\left(u_2^2(1+p_2^2) - u_1^2(1+p_1^2)\right)\right]\\
        &\ \ \geqslant 0 \nonumber
    \end{align}
    }%
    as required. We used the fact that $f_1'(x) \leqslant 0$ to obtain (\ref{scs31}). Equation (\ref{scs3}) follows by substitution. Since $x \in (0,x^*]$, we have $u_ix \leqslant \pi/2$. Also, note that $f_1(x),k_1,k_2,u_1,u_2 > 0$. Then the last inequality follows from the fact that $\alpha\cot(k\alpha)$ is decreasing in $\alpha$ and $\alpha^2(1+\cot(k\alpha)^2)$ is increasing in $\alpha$ for all $\alpha$ such that $\alpha > 0$ and $k\alpha \in (0,\pi/2]$, so that either both terms in the product $\left(u_1p_1-u_2p_2\right)\left(u_2^2(1+p_2^2) - u_1^2(1+p_1^2)\right)$ are non-positive or both non-negative.
    
    \medskip

    Finally, we substitute (\ref{scs2}) in $f_1'(x)f_2(x) - f_1(x)f_2'(x)$, to get
    {
    \begin{align}
        &f_1'(x)f_2(x) - f_1(x)f_2'(x)\nonumber\\
        &\ \ \geqslant f_1'(x)f_2(x) - f_1'(x)\left(k_1\sqrt{p_1^2+1} + k_2\sqrt{p_2^2+1}\right)^2\nonumber\\
        &\ \ = -f_1'(x)\left[-f_2(x) + \left(k_1\sqrt{p_1^2+1} + k_2\sqrt{p_2^2+1}\right)^2\right]\nonumber\\
        &\ \ = -f_1'(x)\left[\left(k_1\sqrt{p_1^2+1}+k_2\sqrt{p_2^2+1}\right)^2 - (k_1p_1+k_2p_2)^2-1\right]\nonumber\\
        &\ \ = -f_1'(x)\left[k_1^2+k_2^2+2k_1k_2\left(\sqrt{p_1^2+1}\sqrt{p_2^2+1}-p_1p_2\right) - 1\right]\nonumber\\
        \label{lem2geq1}
        &\ \ \geqslant -f_1'(x)(k_1^2+k_2^2+2k_1k_2 - 1) \\
        &\ \ = -f_1'(x)((k_1+k_2)^2 - 1)\nonumber\\
        &\ \ \geqslant 0.\nonumber
    \end{align}
    }%
    We used the fact that $\sqrt{a^2+1}\sqrt{b^2+1}-ab \geqslant 1$ to obtain (\ref{lem2geq1}).
\end{proof}

\end{document}